\documentclass[]{amsart}

\usepackage{amsfonts, amsmath, amssymb, amsthm} 
\usepackage{enumerate}
\usepackage{comment}
\usepackage{xfrac}
\usepackage{nicefrac} 
\usepackage[utf8]{inputenc}
\usepackage[T1]{fontenc}

\def\R{\ensuremath{\mathbb{R}}}

\usepackage{graphicx} 
\usepackage{grffile} 
\usepackage{caption}
\usepackage{subcaption} 
\usepackage{tikz, fullpage, longtable, float} 
\usepackage{tikz-cd}
\usepackage{wrapfig}
\usepackage{pgfplots}
\pgfplotsset{compat=1.13}
\usepackage{braids}

\usepackage{centernot}

\theoremstyle{plain}
\newtheorem{thm}{Theorem}[section]

\newtheorem{obs}[thm]{Observation}
\newtheorem{lemma}[thm]{Lemma}

\newtheorem{prop}[thm]{Proposition}
\newtheorem{question}[thm]{Question}

\theoremstyle{definition}
\newtheorem{defn}[thm]{Definition}
\newtheorem{exm}[thm]{Example}

\theoremstyle{remark}
\newtheorem{rmk}[thm]{Remark}

\usepackage[style=numeric]{biblatex}
\usepackage[hidelinks]{hyperref} 
\usepackage{color}
\usepackage{setspace} 

\title{Negative Curvature in Locally Reducible Artin Groups}
\author{Jill Mastrocola}
\date{}

\begin{document}
	
\begin{abstract}
	In this paper, we define the 2-complete Artin complex and show that it is systolic for locally reducible Artin groups. The stabilizers of simplices in this complex are exactly the proper parabolic subgroups which are "2-complete." We use this systolicity to show that parabolic subgroups, with 2-completions that are not the whole Artin group, are weakly malnormal. This allows us to conclude that many locally reducible Artin groups are acylindrically hyperbolic. 
\end{abstract}

\maketitle

\section{Introduction}

Geometric group theory uses geometric and topological methods to study finitely generated groups. If one considers the action of a group on a space satisfying certain geometric properties, one can often recover algebraic properties of the group. Artin groups, one class of particular interest, are closely related to the relatively well-studied Coxeter groups and mapping class groups. Much of the successful study of Artin groups has been through the action of the group on a negatively or nonpositively curved space. In this paper, we study a subclass known as locally reducible Artin groups by defining a new nonpositively curved complex and using it to show many Artin groups within this class are acylindrically hyperbolic.

Artin groups were introduced in the 1960's by Tits as an extension of Coxeter groups. There are many long-standing open questions and conjectures about Artin groups. For example, it is unknown whether Artin groups are torsion-free, have solvable word problem, or satisfy the famous $K(\pi, 1)$-conjecture. Much progress has been made for particular subclasses of Artin groups, but as a whole they remain quite mysterious. 

An \textbf{Artin group} is a group with a presentation of the form $$\langle s_1, \dotsc, s_n \ |\ \underbrace{s_i s_j s_i\dotsc}_{m_{ij}} = \underbrace{s_j s_i s_j \dotsc}_{m_{ij}} \rangle$$ for some finite set of generators $S = \{s_1, \dotsc, s_n\}$ where $m_{ij}\in\{2, 3, \dotsc, \infty \}$ for each $1\leq i < j \leq n$. If $m_{ij} = \infty$, then $s_i$ and $s_j$ have no relation. We can encode the same information in a \textbf{defining graph}, denoted $\Gamma$, with vertices $V(\Gamma) = S$. If $m_{ij}<\infty$, then $\Gamma$ has an edge between vertices $s_i$ and $s_j$ labeled by the integer $m_{ij}$. An Artin group with defining graph $\Gamma$ is denoted $A_\Gamma$. 

To each Artin group there is an associated \textbf{Coxeter group}, $W_\Gamma$, which is the quotient of $A_\Gamma$ defined by adding the relations $s_i^2 = 1$ for all $1\leq i \leq n$. If $W_\Gamma$ is a finite group, we say that $A_\Gamma$ is \textbf{of finite type}. Note that $A_\Gamma$ will always be an infinite group. 

For a subset of generators $T\subseteq S$, the associated \textbf{standard parabolic subgroup} of $A_\Gamma$ is the subgroup generated by $T$. Van der Lek \cite{van_der_lek_homotopy_1983} proved that this subgroup is isomorphic to the Artin group determined by the full subgraph of $\Gamma$ with vertices $T$, which we will denote $A_T$. Any conjugate of such a subgroup is called a \textbf{parabolic subgroup}. 

Among the commonly studied subclasses of Artin groups, finite type Artin groups and right-angled Artin groups (those for which $m_{ij} = 2$ or $\infty$ for all $i,j$) have seen the most progress. In the finite type case, an Artin group admits a Garside structure, which provides normal forms and combinatorial structure for studying the group. Right-angled Artin groups are examples of graph products of groups. They have algorithmic properties and admit a CAT(0) cubical structure. 

The class of 2-dimensional Artin groups has also seen significant progress. (The dimension of an Artin group is the maximum cardinality of a subset $T\subseteq V(\Gamma)$ such that $A_T$ is finite type.) By groundbreaking work of Charney and Davis, 2-dimensional Artin groups are torsion-free and satisfy the $K(\pi,1)$-conjecture \cite{charney_davis_1995}. They are also known to have solvable word problem by work of Chermak \cite{chermak_1998}. More recently Cumplido, Martin, and Vaskou have developed a rich study of large-type Artin groups (those for which $m_{ij}\geq 3$ for all $i,j$) and have proven many long-standing conjectures about them: the intersection of a family of parabolic subgroups is itself a parabolic subgroup \cite{cmv_large_2023}, automorphisms are classified, and \cite{vaskou_automs_large} the class of large-type Artin groups is rigid under isomorphism \cite{martin_vaskou_isom_problem_large}.

Our focus in this paper is the following class of Artin groups. They were introduced by Charney, who showed they satisfy both the $K(\pi, 1)$-Conjecture and the Tits Conjecture \cite{charney_tits_2000}. 

\begin{defn}
	An Artin group $A_\Gamma$ is \textbf{locally reducible} if the only finite type triangles contained in $\Gamma$ are of the form $2-2-k$ for some $k\geq 2$. 
\end{defn}

Within this class of Artin groups, the only finite type parabolic subgroups which occur are cyclic subgroups, dihedral subgroups (i.e., $A_T$ for $T = \{a,b\}$ where $m_{ab} \geq 3$), and direct products of cyclic and dihedral subgroups.

The notion of hyperbolic groups was introduced by Gromov in the 1980's \cite{gromov_hyperbolic_1987}, and sparked much of the field of geometric group theory as we know it today. The action of a group on a hyperbolic or nonpositively curved space has proven to be a very useful tool in the field, especially when the action is proper and cocompact. More recently, there have been many attempts at generalizing hyperbolic groups. In this paper, we focus on acylindrical hyperbolicity. 

A group action of $G$ on a metric space $(X,d)$ is \textbf{acylindrical} if for every $\epsilon\geq0$, there exist constants $R\geq 0$ and $N\geq 0$ such that for all points $x,y\in X$ such that $d(x,y) \geq R$, $|\{ g\in G\ |\ d(x,gx)\leq \epsilon \text{\ and\ } d(y, gy)\leq \epsilon \}|\leq N$. A group is \textbf{acylindrically hyperbolic} if it is not virtually cyclic and has an acylindrical action on a hyperbolic space. 

Informally, acylindricity can be interpreted as a sort of properness on pairs of points that are far apart in the space. The original definition of acylindrical hyperbolicity goes back to Sela \cite{sela_acyl}, which applies to group actions on trees. The more general and current definition is due to Bowditch \cite{bowditch_acy_hyp} and was shown by Osin \cite{osin_acy_hyp} to be equivalent to other conditions studied in \cite{bestvina_fujiwara_bounded_cohom, hamenstadt_bounded_cohom, sisto_contracting_elts}.

Acylindrically hyperbolic groups include all but finitely many mapping class groups \cite{bowditch_acy_hyp, masur_minsky}, $Out(F_n)$ for $n\geq 2$ \cite{bestvina_feighn}, and many CAT(0) groups \cite{sisto_contracting_elts}. Certain classes of Artin groups are known to be acylindrically hyperbolic, including the following classes, provided the group is irreducible, not cyclic, and of sufficient rank: 
\begin{itemize}
	\item right-angled Artin groups \cite{caprace_sageev, osin_acy_hyp, sisto_contracting_elts}
	\item Artin groups whose defining graph is not a join \cite{charney_morriswright_acy_hyp, chatterji_martin}
	\item XXL Artin groups (all $m_{ij}\geq 5$) \cite{haettel_xxl}
	\item triangle-free Artin groups \cite{kato_oguni_tri_free}
	\item Artin groups of euclidean type \cite{calvez_euclidean_acy_hyp}
	\item 2-dimensional Artin groups \cite{martin_przytycki_acy_hyp, vaskou_acy_hyp_2diml}
\end{itemize}

In this paper, we show that many locally reducible Artin groups are acylindrically hyperbolic. For $A_\Gamma$ an arbitrary Artin group, define $\widehat{\Gamma}$ to be the graph obtained from $\Gamma$ by deleting all edges not labeled by 2. We consider a single vertex of $\widehat{\Gamma}$ with no edges to be its own connected component. A parabolic subgroup of $A_\Gamma$ is \textbf{2-complete} if it can be written as $gA_Tg^{-1}$ for some $g\in A_\Gamma$ and some $T\subseteq V(\Gamma)$ where $T$ is a union of connected components of $\widehat{\Gamma}$. A \textbf{2-completion} of a parabolic subgroup $P$ is a 2-complete parabolic subgroup which contains $P$. Note that 2-completions are not unique (see Remark \ref{canonical 2-completion}).

\begin{thm}\label{acy hyp} Let $A_\Gamma$ be an Artin group. If either
	\begin{enumerate} 
		\item $A_\Gamma$ is locally reducible and has a maximal finite-type subgroup which is dihedral and which has a 2-completion that is not all of $A_\Gamma$, or
		\item $A_\Gamma$ splits as an amalgamated product $A_{\Gamma_1} \ast_{A_{\Gamma_1 \cap \Gamma_2}} A_{\Gamma_2}$ such that $A_{\Gamma_1}$ is locally reducible and there is a 2-completion of $A_{\Gamma_1 \cap \Gamma_2}$ which does not contain all of $A_{\Gamma_1}$,
	\end{enumerate}
then $A_\Gamma$ is acylindrically hyperbolic.
\end{thm}

The following theorem is the key ingredient to this result. This allows us to apply work of Martin \cite{martin_acy_hyp} and Minasyan and Osin \cite{minasyan_osin_acy_hyp_amalg} to conclude that many Artin groups which are locally reducible are acylindrically hyperbolic. 

\begin{thm}\label{weakly malnormal}
	Let $A_\Gamma$ be a locally reducible Artin group such that $\widehat{\Gamma}$ has at least two connected components. Then any parabolic subgroup $P$ of $A_\Gamma$ which has a 2-completion that is not all of $A_\Gamma$ is weakly malnormal (i.e., there exists $g\in G$ such that $|P\cap gPg^{-1}|<\infty$). 
\end{thm}

As is a common theme in the study of Artin groups, nonpositively curved geometric complexes are a key component of these results. For locally reducible Artin groups, it is known that the Deligne complex is CAT(0) \cite{charney_tits_2000}. In Section \ref*{Defn of the 2-complete Artin Complex Section}, we introduce the 2-complete Artin complex and show that it satisfies a combinatorial version of nonpositive curvature known as systolicity. 

Intuitively, a geodesic metric space is said to be CAT(0) if its triangles are “at least as thin as” comparison triangles in $\R^2$. In the case of a cubical complex with the standard Euclidean metric, there is a simple link condition to check that the complex is CAT(0). However, when dealing with a different metric on a cube complex, or in the setting of a simplicial complex, CAT(0)-ness can be quite hard to check, especially in high dimensions. Systolicity was introduced as an analogue of CAT(0)-ness for simplicial complexes by Januszkiewicz and \'{S}wi\k{a}tkowski \cite{januszkiewicz_swiatkowski_2006} and independently by Haglund \cite{haglund_systolicity}. Systolicity shares many of the same nice consequences of CAT(0)-ness, and it can always be checked in a combinatorial way. See Section \ref{2-complete Artin Complex is Systolic Section} for a precise definition. The following result is the key to proving Theorem \ref{weakly malnormal}. 

\begin{thm}\label{2-complete Artin cplx is systolic}
	Let $A_\Gamma$ be a locally reducible Artin group. If there are at least three connected components in $\widehat{\Gamma}$, then the 2-complete Artin complex of $A_\Gamma$ is systolic. 
\end{thm}

We begin by recalling the definitions of the well-known Delgine and Artin complexes (Section \ref{Deligne and Artin Complexes Section}). We then define the 2-complete Artin complex (Section \ref{Defn of the 2-complete Artin Complex Section}) and show that it is systolic (Section \ref{2-complete Artin Complex is Systolic Section}). We prove Theorem \ref{weakly malnormal} (Section \ref{Weakly Malnormal Section}) and apply theorems of Martin and Minasyan-Osin to show that many locally reducible Artin groups are acylindrically hyperbolic (Section \ref{Acy Hyp Section}). Finally, we discuss some potential future directions related to this work (Section \ref{Future Directions Section}).

\subsection*{Acknowledgements}

I would like to thank my advisor Ruth Charney, who introduced me to the subject of geometric group theory and has supported me throughout my graduate school journey. None of this work would have been possible without her guidance and advice along the way. I also want to thank Alex Martin, who offered helpful motivations and suggestions early on, and has supported my work throughout. Special thanks go to Carolyn Abbott and Jason Behrstock, who both provided helpful feedback on my results and my writing.

\section{Complexes Associated to Artin Groups}\label{Deligne and Artin Complexes Section}

\subsection{The Deligne Complex} 
The Deligne complex for an arbitrary Artin group $A_\Gamma$ with generating set $S$ was introduced by Charney and Davis \cite{charney_davis_1995}, generalizing the work of Deligne \cite{deligne_1972}, who defined such a complex for finite type Artin groups. Let $\mathcal{S}^f = \{T\subseteq S\ |\ A_T \text{ is finite type} \}$. We include $\emptyset \in \mathcal{S}^f$ with the standard convention that $A_\emptyset = \{ 1\}$. Note that $\mathcal{S}^f$ is a partially ordered set with respect to inclusion. 
	
\begin{defn}
	The \textbf{Deligne complex} of $A_\Gamma$, denoted $D_\Gamma$, is the cubical complex whose vertices correspond to cosets $gA_T$ for $g\in A_\Gamma$ and $T\in \mathcal{S}^f$ and whose $k$-dimensional cubes correspond to intervals of inclusion $[gA_R, gA_T]$ of length $k+1$. 
\end{defn}

The Artin group $A_\Gamma$ acts by left multiplication on $D_\Gamma$. The stabilizer of a vertex $gA_T$ is the parabolic subgroup $gA_T g^{-1}$. A fundamental domain for the action is the subcomplex $K_{D_\Gamma}$ of $D_\Gamma$ spanned by vertices of the form $A_T$ for $T\in \mathcal{S}^f$. 

We will now define the Moussong metric for the Deligne complex. We will start by describing the Moussong metric in its original setting: the Davis complex \cite{moussong_1988}. This complex associated to the Coxeter group $W_\Gamma$ is defined analogously to the Deligne complex. We consider the same set $\mathcal{S}^f$ as above. 

\begin{defn}
	The \textbf{Davis complex} of $W_\Gamma$, denoted $C_\Gamma$, is the cubical complex whose vertices correspond to cosets $wW_T$ for $w\in W_\Gamma$ and $T\in \mathcal{S}^f$ and whose $k$-dimensional cubes correspond to intervals of inclusion $[wW_R, wW_T]$ of length $k+1$. 
\end{defn}

Similarly to above, $W_\Gamma$ acts by left multiplication on $C_\Gamma$ and has a fundamental domain $K_{C_\Gamma}$ spanned by vertices of the form $W_T$. 

When $W_\Gamma$ is finite, it can be viewed as a group of orthogonal transformations of $\R^n$. The generators $s\in S$ act as reflections across the walls of a simplicial cone $Z\subseteq \R^n$. There is a unique point $x_\emptyset \in Z\cap \mathbb{S}^{n-1}$ which is equidistant from each wall of $Z$. The convex hull $X$ of the orbit of $x_\emptyset$ under the action of $W_\Gamma$ is isomorphic to the Davis complex $C_\Gamma$. The origin in $\R^n$ corresponds to the vertex $W_S$ in $C_\Gamma$ and $x_\emptyset$ in $\R^n$ corresponds to the vertex $W_\emptyset$ in $C_\Gamma$. The fundamental domain $K_{C_\Gamma}$ is identified with the intersection $X\cap Z$. This isomorphism gives a metric on $C_\Gamma$; in this case $C_\Gamma$ is called the \textbf{Coxeter cell} of $W_\Gamma$. 

If $W_\Gamma$ is infinite, we can cover $C_\Gamma$ by the closed stars of vertices of the form $wW_T$ for $T$ maximal elements of $\mathcal{S}^f$. Since $W_T$ is finite, a star corresponding to a vertex of this form is isomorphic to the Coxeter cell of $W_T$ and can be given the metric of that Coxeter cell. These Coxeter cells together give a piecewise Euclidean metric on $C_\Gamma$, called the \textbf{Moussong metric}. Moussong proved that $C_\Gamma$ with this metric is CAT(0) for all $W_\Gamma$ \cite{moussong_1988}.

For an Artin group $A_\Gamma$, the quotient map to its associated Coxeter group $A_\Gamma\to W_\Gamma$ induces an equivariant projection $D_\Gamma \to C_\Gamma$, which is an isomorphism when restricted to the fundamental domain $K_{D_\Gamma}$. We define the Moussong metric on the Deligne complex $D_\Gamma$ to be the piecewise Euclidean metric defined on each translate of $K_{D_\Gamma}$ by letting $K_{D_\Gamma}$ be isometric to $K_{C_\Gamma}$. 

\begin{thm}\cite[Theorem 3.2]{charney_tits_2000} \label{deligne cplx CAT(0) for LR}
If $A_\Gamma$ is locally reducible, then $D_\Gamma$ with the Moussong metric is CAT(0). 
\end{thm}

\subsection{The Artin Complex}

When Charney and Davis defined the Deligne complex in \cite{charney_davis_1995}, they also introduced the analogous Artin complex.

\begin{defn}
	The \textbf{Artin complex} associated to $A_\Gamma$, denoted $X_\Gamma$, is the simplicial complex with vertices corresponding to cosets of the form $gA_{S\setminus \{t\}}$ for $g\in A_\Gamma$ and $t\in S$. A set of vertices spans a simplex if the corresponding cosets have nonempty intersection. 
\end{defn}

Cumplido, Martin, and Vaskou showed that the Artin complex is systolic for Artin groups of large type \cite{cmv_large_2023}. Using the consequences of this geometry, they show that the intersection of parabolic subgroups of a large-type Artin group is a parabolic subgroup, solving a long-standing conjecture. The Artin complex was also studied by Godelle and Paris, who in particular showed that it is a flag complex \cite{godelle_paris_2012}. Their work has inspired a new approach to the $K(\pi, 1)$-conjecture by Huang \cite{huang_Kpi1_2023}. In this paper we define a modification of this complex, called the 2-complete Artin complex.

\section{The $2$-Complete Artin Complex}\label{Defn of the 2-complete Artin Complex Section}

Let $A_\Gamma$ be an Artin group with defining graph $\Gamma$, and let $v$ be a vertex in $\Gamma$. The \textbf{2-completion} of $v$ is the full subgraph of $\Gamma$ spanned by $v$ together with all vertices which can be reached from $v$ by an edge path labeled only by 2's. 

Let $\mathcal{S}^{(2)} = \{T\subseteq S\ |\ \text{if } v\in T \text{, then the 2-completion of } v \text{ is contained in } T \}$. This is a partially ordered set with respect to inclusion. If $\Gamma$ is as in Example \ref{2-complete Artin complex example}, then $T = \{a,b,c\}\in \mathcal{S}^{(2)}$ but $W = \{a,b,c,d\}\neq \mathcal{S}^{(2)}$. 

Let $\widehat{\Gamma}$ be the subgraph of $\Gamma$ which consists of all the vertices of $\Gamma$ and the edges of $\Gamma$ labeled by $2$. Connected components of $\widehat{\Gamma}$ correspond to non-empty, minimal elements of $\mathcal{S}^{(2)}$. We consider a single vertex in $\Gamma$ which has no edges labeled $2$ to be its own connected component in $\widehat{\Gamma}$. (See Example \ref{2-complete Artin complex example}.)

A subset $T\subseteq V(\Gamma)$ of generators is \textbf{2-complete} if it is an element of $\mathcal{S}^{(2)}$, i.e., if for every $v\in T$, all the vertices in the 2-completion of $v$ are also in $T$. A parabolic subgroup is 2-complete if it can be written in the form $gA_Tg^{-1}$ for some $g\in A_\Gamma$ and some 2-complete $T\subseteq V(\Gamma)$. 

A \textbf{2-completion} of a parabolic subgroup $P$ of $A_\Gamma$ is a parabolic subgroup that is 2-complete and contains $P$. 

\begin{rmk}\label{canonical 2-completion}
	A parabolic subgroup can have many distinct 2-completions. For a standard parabolic subgroup $A_T$, there is a canonical choice: $A_R$ where $R$ is the union of the 2-completions of the vertices of $T$. However in the case of a general (non-standard) parabolic subgroup, there is not a canonical choice of 2-completion since a parabolic subgroup may be conjugate to more than one standard parabolic.
\end{rmk}

\begin{defn}
	The \textbf{2-complete Artin complex}, $\widehat{X_\Gamma}$, is the following simplicial complex of groups: 
	\begin{itemize}
		\item there is a vertex for every left coset of a standard parabolic subgroup of the form $A_{\Gamma \setminus T}$ for $T$ some single connected component of $\widehat{\Gamma}$, and 
		\item a collection of vertices spans a simplex if the associated cosets have collective nonempty intersection.
	\end{itemize}
\end{defn}

Let $K$ denote the fundamental domain of $\widehat{X_\Gamma}$, i.e., the simplex corresponding to standard parabolic subgroups of $A_\Gamma$. 

\begin{exm}\label{2-complete Artin complex example}
	We will now show an example of $\widehat{\Gamma}$ and $K$, given the following defining graph $\Gamma$. 

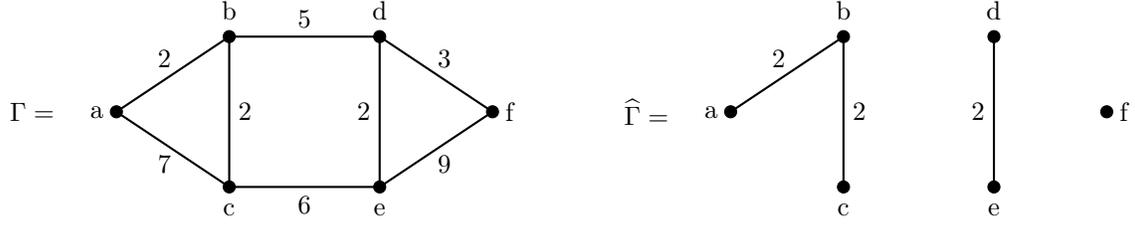
	\begin{center}
	\begin{tikzpicture}
	\def \lilRad {0.08}
	\draw [color=black,  style=thick] (0,0)--(0,2);
	\draw [color=black,  style=thick] (0,0)--(2,0);
	\draw [color=black,  style=thick] (0,2)--(2,2);
	\draw [color=black,  style=thick] (2,0)--(2,2);
	\draw [color=black,  style=thick] (-1.5,1)--(0,0);
	\draw [color=black,  style=thick] (-1.5,1)--(0,2);
	\draw [color=black,  style=thick] (3.5,1)--(2,0);
	\draw [color=black,  style=thick] (3.5,1)--(2,2);
	\draw [fill=black] (-1.5,1) circle (\lilRad);
	\draw [fill=black] (0,0) circle (\lilRad);
	\draw [fill=black] (0,2) circle (\lilRad);
	\draw [fill=black] (2,0) circle (\lilRad);
	\draw [fill=black] (2,2) circle (\lilRad);
	\draw [fill=black] (3.5,1) circle (\lilRad);
	\node [left] at (-1.55,1) {a}; 
	\node [above] at (0,2.1) {b}; 
	\node [below] at (0,-0.1) {c}; 
	\node [above] at (2,2.1) {d}; 
	\node [below] at (2,-0.1) {e}; 
	\node [right] at (3.55,1) {f};
	\node [left] at (-0.65,1.7) {2};
	\node [left] at (-0.65,0.3) {7};
	\node [right] at (0,1) {2};
	\node [left] at (2,1) {2};
	\node [above] at (1,2) {5};
	\node [below] at (1,0) {6};
	\node [right] at (2.65,1.7) {3};
	\node [right] at (2.65,0.3) {9};
	\node [left] at (-2.2,1) {$\Gamma = $};
	\end{tikzpicture}
	\hspace{1cm}
	\begin{tikzpicture}
	\def \lilRad {0.08}
	\draw [color=black,  style=thick] (0,0)--(0,2);
	\draw [color=black,  style=thick] (2,0)--(2,2);
	\draw [color=black,  style=thick] (-1.5,1)--(0,2);
	\draw [fill=black] (-1.5,1) circle (\lilRad);
	\draw [fill=black] (0,0) circle (\lilRad);
	\draw [fill=black] (0,2) circle (\lilRad);
	\draw [fill=black] (2,0) circle (\lilRad);
	\draw [fill=black] (2,2) circle (\lilRad);
	\draw [fill=black] (3.5,1) circle (\lilRad);
	\node [left] at (-1.55,1) {a}; 
	\node [above] at (0,2.1) {b}; 
	\node [below] at (0,-0.1) {c}; 
	\node [above] at (2,2.1) {d}; 
	\node [below] at (2,-0.1) {e}; 
	\node [right] at (3.55,1) {f};
	\node [left] at (-0.65,1.7) {2};
	\node [right] at (0,1) {2};
	\node [left] at (2,1) {2};
	\node [left] at (-2.2,1) {$\widehat{\Gamma} = $};
	\end{tikzpicture}
	\end{center}

Label the connected components in $\widehat{\Gamma}$ as $T = \{a, b, c\}$, $U = \{d, e\}$, and $V=\{f\}$. The 2-complete Artin complex $\widehat{X_\Gamma}$ has the following fundamental domain: 

	
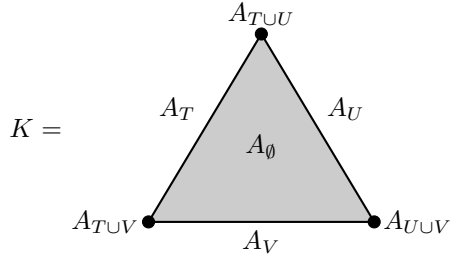
\begin{center}
	\begin{tikzpicture}
	\def \lilRad {0.08}
	\filldraw[color=black, fill=gray!40, thick] ((0,0) -- (3,0) -- (1.5,2.5) -- cycle;
	\draw [fill=black] (0,0) circle (\lilRad);
	\draw [fill=black] (3,0) circle (\lilRad);
	\draw [fill=black] (1.5,2.5) circle (\lilRad);
	\node at (1.5,1) {$A_\emptyset$};
	\node [left] at (0,0) {$A_{T\cup V}$};
	\node [right] at (3,0) {$A_{U\cup V}$};
	\node [above] at (1.5,2.5) {$A_{T\cup U}$};
	\node [left] at (0.75,1.5) {$A_T$};
	\node [right] at (2.25,1.5) {$A_U$};
	\node [below] at (1.5,0) {$A_V$};
	\node [below] at (-1.5,1.5) {$K = $};
	\end{tikzpicture}
	\end{center}

\end{exm}

\begin{obs}
	The $2$-complete proper parabolic subgroups of $A_\Gamma$ are exactly the stabilizers of simplices of $\widehat{X_\Gamma}$. 
\end{obs}

We say a vertex of $\widehat{X_\Gamma}$ has \textbf{type $\Gamma\setminus T$} if it corresponds to coset of the form $gA_{\Gamma\setminus T}$. Likewise, the type of each simplex in $\widehat{X_\Gamma}$ is given by the subset of $\Gamma$ corresponding to its associated coset. Equivalently, the type of a simplex is determined by the intersection of the types of the vertices of that simplex. Note that each top-dimensional simplex corresponds to a coset of the form $gA_\emptyset$. 

\begin{prop}[Basic Properties of $\widehat{X_\Gamma}$]\label{basic_props_2complete_artin_cplx}
	Let $A_\Gamma$ be an Artin group. 
	\begin{enumerate}
		\item If $\widehat{\Gamma}$ has at least two connected components, then $\widehat{X_\Gamma}$ is connected. 
		\item If $\widehat{\Gamma}$ has at least three connected components, then $\widehat{X_\Gamma}$ is simply connected. 
		\item The link in $\widehat{X_\Gamma}$ of a simplex of type $U$ is isomorphic to the 2-complete Artin complex $\widehat{X_{\Gamma'}}$ where $\Gamma'$ is the subgraph of $\Gamma$ spanned by the vertices of $U$. 
	\end{enumerate}
\end{prop}

\begin{proof}
	The 2-complete Artin complex $\widehat{X_\Gamma}$ is a simplicial complex of groups in the language of Bridson-Haefliger \cite{bridson_haefliger}. The fundamental domain $K$ has vertices of the form $A_{\Gamma\setminus T}$ for $T$ a single connected component of $\widehat{\Gamma}$, and edges between vertices corresponding to nonempty intersection. The morphisms associated to the inclusion of faces of $K$ are the natural inclusions of the respective Artin groups. In this language, $\widehat{X_\Gamma}$ is the development of $K$, in the sense of \cite{bridson_haefliger}. 
	
	(1): If there are at least two connected components of $\widehat{\Gamma}$, then the fundamental domain $K$ of $\widehat{X_\Gamma}$ is at least one-dimensional and is connected. Furthermore, the Artin group $A_\Gamma$ is generated by the collection of parabolic subgroups which are stabilizers of simplices of $\widehat{X_\Gamma}$. By \cite[Prop 12.20]{bridson_haefliger}, $\widehat{X_\Gamma}$ is connected. 
	
	(2): If there are at least three connected components of $\widehat{\Gamma}$, then by \cite[Prop 12.20 (4)]{bridson_haefliger}, $\widehat{X_\Gamma}$ is simply connected. 
	
	(3): By \cite[Construction 12.24]{bridson_haefliger}, we can describe the link of a simplex in $\widehat{X_\Gamma}$ as the development of an appropriate subcomplex of groups. 
\end{proof}

\section{The 2-complete Artin Complex is Systolic}\label{2-complete Artin Complex is Systolic Section}
	
In this section we show that for a locally reducible Artin group $A_\Gamma$, where $\widehat{\Gamma}$ has at least three connected components, the associated 2-complete Artin complex $\widehat{X_\Gamma}$ is systolic. Let $X$ be a simplicial complex. A subcomplex $Y\subseteq X$ is \textbf{full} if any simplex of $X$ spanned by a set of vertices of $Y$ is a simplex of $Y$. A \textbf{full cycle} $\gamma\subseteq X$ is a cycle that is full as a subcomplex of $X$. See below for some examples. The \textbf{systole} of a simplicial complex $X$ is $$\text{sys}(X) = \text{min} \{ |\gamma| \mid  \gamma \text{ is a full cycle in } X\}$$
For a simplicial complex $X$, $\text{sys}(X)\geq 3$, and if there are no full cycles in $X$ then $\text{sys}(X) = \infty$. A simplicial complex $X$ is \textbf{locally $k$-large} if $\text{sys}(lk_X(Y))\geq k$ for all simplices $Y\subseteq X$. $X$ is \textbf{$k$-large} if it is locally $k$-large and $\text{sys}(X)\geq k$. 

\begin{defn}
		A simplicial complex $X$ is \textbf{systolic} if it is connected, simply-connected, and locally $6$-large. 
\end{defn}

\begin{center}
	\begin{tikzpicture}
	\draw [line width = 4pt, yellow] (1,0) -- (1,1) -- (0,1) -- cycle;
	\filldraw[color=black, fill=gray!40, thick] ((1,0) -- (1,1) -- (0,1) -- cycle;
	\draw node[below] at (0.5,-0.25) {Not a full cycle}; 
	\end{tikzpicture}		\begin{tikzpicture}
	\draw [line width = 4pt, yellow] (0,0) -- (1,0) -- (1,1) -- (0,1) -- cycle;
	
	\filldraw (0,0) circle [radius = 0.04];
	\filldraw (1,0) circle [radius = 0.04];
	\filldraw (0,1) circle [radius = 0.04];
	\filldraw (1,1) circle [radius = 0.04];
	
	\draw [thick] (0,0) -- (1,0) -- (1,1) -- (0,1) -- cycle;
	\draw [thick] (1,0) -- (0,1);
	
	\draw node[below] at (0.5,-0.25) {Not a full cycle}; 
	\end{tikzpicture}		\begin{tikzpicture}[scale = 1]
	\draw [line width = 4pt, yellow] (0,0) -- (1,0) -- (1,1) -- (0,1) -- cycle;
	
	\filldraw (0,0) circle [radius = 0.04];
	\filldraw (1,0) circle [radius = 0.04];
	\filldraw (0,1) circle [radius = 0.04];
	\filldraw (1,1) circle [radius = 0.04];
	
	\draw [thick] (0,0) -- (1,0) -- (1,1) -- (0,1) -- cycle;
	\draw node[below] at (0.5,-0.25) {A full cycle}; 
	\end{tikzpicture}
\end{center}

\begin{thm} \cite[Theorem 2.11]{bridson_haefliger} \label{flat quad thm}
	The sum of the interior angles of a quadrilateral in a CAT(0) space cannot be greater than $2\pi$. 
\end{thm}

\begin{lemma}\cite[Lemma 5.1]{charney_tits_2000} \label{Deligne complex embeds}
	Let $A_\Gamma$ be a locally reducible Artin group. For any $T\subseteq S$, the natural inclusion $D_T\hookrightarrow D_\Gamma$ is an isometric embedding, and hence the image is convex. 
\end{lemma}

\begin{prop}\label{connected components of Gamma hat}
	Let $A_\Gamma$ be locally reducible, and let $R\subseteq V(\Gamma)$ correspond to a finite type clique in $\Gamma$. If $|R|\geq 3$, then $R$ must be contained in a single connected component of $\widehat{\Gamma}$. 
\end{prop}

\begin{proof}
	Choose $r\in R$. For any vertices $u,v\in R\setminus \{r\}$, the collection $\{r, u, v\}$ spans a finite-type triangle in $\Gamma$. Since $A_\Gamma$ is locally reducible, this triangle has type $2-2-k$. Thus $u$ and $v$ must lie in the 2-completion of $r$, and hence all three vertices must be in the same connected component of $\widehat{\Gamma}$. 
\end{proof}

\begin{lemma}\label{TU not equal UT}
	Let $A_\Gamma$ be a locally reducible Artin group such that $\widehat{\Gamma}$ has at least two connected components. Let $T$ and $U$ be two distinct connected components of $\widehat{\Gamma}$. If $t_1, t_2\in A_T$ and $u_1, u_2\in A_U$ are pairs of (not necessarily distinct) elements, then $t_1u_1\neq u_2t_2$. 
\end{lemma}

\begin{proof}
	Suppose by way of contradiction that there are elements $t_1, t_2\in A_T$ and $u_1, u_2\in A_U$ such that $t_1 u_1 = u_2 t_2$. Then there is a quadrilateral in the Deligne complex $D_\Gamma$ determined by vertices $A_\emptyset$, $u_1 A_\emptyset$, $t_1 u_1 A_\emptyset = u_2 t_2 A_\emptyset$, $t_2 A_\emptyset$, and the geodesics between them. 
	
	
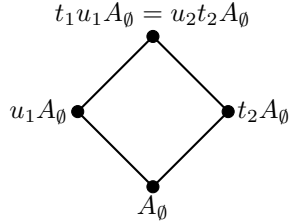
\begin{center}
		\begin{tikzpicture}
		\def \lilRad {0.08}
		\draw [color=black,  style=thick] (1,0)--(0,1);
		\draw [color=black,  style=thick] (1,0)--(2,1);
		\draw [color=black,  style=thick] (0,1)--(1,2);
		\draw [color=black,  style=thick] (2,1)--(1,2);
		\draw [fill=black] (1,0) circle (\lilRad);
		\draw [fill=black] (0,1) circle (\lilRad);
		\draw [fill=black] (2,1) circle (\lilRad);
		\draw [fill=black] (1,2) circle (\lilRad);
		\node [below] at (1,0) {$A_\emptyset$};
		\node [left] at (0,1) {$u_1 A_\emptyset$};
		\node [right] at (2,1) {$t_2 A_\emptyset$};
		\node [above] at (1,2) {$t_1u_1 A_\emptyset = u_2 t_2 A_\emptyset$};
		\end{tikzpicture}
	\end{center}
	
	The geodesics forming the sides of this quadrilateral are not edges, and do not even need to be edge paths. However since $D_\Gamma$ is CAT(0) by Theorem \ref{deligne cplx CAT(0) for LR}, the geodesics are unique. 
	
	We will now consider the measure of each interior angle of the quadrilateral. The goal is to show that such a quadrilateral cannot exist using Theorem \ref{flat quad thm}. 
	
	By Lemma \ref{Deligne complex embeds}, the two sides $[A_\emptyset, u_1 A_\emptyset]$ and $[A_\emptyset, t_2 A_\emptyset]$ are contained in the Deligne complex associated to $A_U$ (denoted $D_U$) and the Deligne complex associated to $A_T$ (denoted $D_T$), respectively. Hence the angle between the two sides, i.e., the interior angle of the quadrilateral at $A_\emptyset$, is bounded below by the angle at $A_\emptyset$ between $D_U$ and $D_T$ as embedded subcomplexes of $D_\Gamma$. 
	
	Let $v_s$ denote the vertex in $lk_{D_\Gamma}(A_\emptyset)$ corresponding to the edge $(A_\emptyset, A_{\{s\}})$ in $D_\Gamma$. The angle between $D_U$ and $D_T$ in $D_\Gamma$ is determined by the distance between vertices $v_t$ and $v_u$ in $lk_{D_\Gamma}(A_\emptyset)$ for all $t\in T$ and $u\in U$. 
	
	Note that the 1-skeleton of $lk_{D_\Gamma}(A_\emptyset)$ is isomorphic to the defining graph $\Gamma$. Indeed, there is one vertex in $lk_{D_\Gamma}(A_\emptyset)$ for each edge in $D_\Gamma$ of the form $(A_\emptyset, A_{\{s\}})$, and hence for each generator $s\in V(\Gamma)$. And there is an edge between two vertices $v_r$ and $v_s$ in $lk_{D_\Gamma}(A_\emptyset)$ if and only if there is a 2-dimensional cube in $D_\Gamma$ with vertices $A_\emptyset$, $A_{\{r\}}$, $A_{\{s\}}$, and $A_{\{r,s\}}$, which happens precisely when $A_{\{r,s\}}$ is a finite type Artin group, i.e., there is an edge between $r$ and $s$ in $\Gamma$. 
	
	Let $t\in T$ and $u\in U$ be arbitrary. By Proposition \ref{connected components of Gamma hat}, for any generator $s\in V(\Gamma)$, it must be that $A_{\{s,t,u\}}$ is not finite type. Hence in $lk_{D_\Gamma}(A_\emptyset)$, there are no $n$-dimensional simplices between $v_t$ and $v_u$ for $n\geq 2$, and so the distance between $v_t$ and $v_u$ is determined by the $1$-skeleton of $lk_{D_\Gamma}(A_\emptyset)$. For an edge $(r,s)$ in $\Gamma$ with label $m_{rs}$, the corresponding edge $(v_r, v_s)$ in $lk_{D_\Gamma} (A_\emptyset)$ has length $\pi - \frac{\pi}{m_{rs}}$ with respect to the Moussong metric. Since $m_{tu}\geq 3$, the distance between $v_t$ and $v_u$ is at least $\frac{2\pi}{3}$. 
	
	Therefore, the interior angle of the above quadrilateral at vertex $A_\emptyset$ is at least $\frac{2\pi}{3}$. 
 
	We can repeat this argument for vertices $u_1 A_\emptyset$ and $t_2 A_\emptyset$ by translating each by $u_1^{-1}$ and $t_2^{-1}$, respectively, which preserves angle measures. 
	
	It follows that the sum of interior angles in the quadrilateral at vertices $u_1 A_\emptyset$, $A_\emptyset$, and $t_2 A_\emptyset$, is at least $2\pi$. By Theorem \ref{flat quad thm}, it is impossible for the fourth vertex of this quadrilateral to exist. Hence we have reached a contradiction, and it must be that $t_1u_1 \neq u_2t_2$. 
\end{proof}

\begin{lemma}\label{2 conn comps 6 large}
	Let $A_\Gamma$ be a locally reducible Artin group. If $\widehat{\Gamma}$ has exactly two connected components, then $\widehat{X_\Gamma}$ is 6-large. 
\end{lemma}
	
\begin{proof}
	If there are exactly two connected components of $\widehat{\Gamma}$, the fundamental domain $K$ of $\widehat{X_\Gamma}$ is a 1-dimensional simplex. By construction, there will be no higher-dimensional simplices in $\widehat{X_\Gamma}$. Thus $\widehat{X_\Gamma}$ is 6-large if $\text{sys}(\widehat{X_\Gamma})\geq 6$. 
	
	Let $T$ and $U$ be the subsets of generators corresponding to the two connected components of $\widehat\Gamma$. Any edge in $\widehat{X_\Gamma}$ will have exactly one vertex of type $T$ and one vertex of type $U$, since every edge in the complex is some translate of $K$. In particular, any cycle in $\widehat{X_\Gamma}$ will have even length. 
	
	We need to show that $\text{sys}(\widehat{X_\Gamma}) \geq 6$. By the preceding paragraph, this is equivalent to showing that there are no full cycles of length 4 in $\widehat{X_\Gamma}$. Consider a path of length 3 in $\widehat{X_\Gamma}$. By translating, we can assume that the fundamental domain $K$ is the middle edge in the path. 

	Since the local group of $K$ is $A_\emptyset$, the local groups of the two edges adjacent to $K$ in the path are of the form $g_1A_\emptyset$ and $g_2A_\emptyset$ for some $g_1, g_2\in A_\Gamma$.
	
	Since $\widehat{X_\Gamma}$ has an edge between $g_2 A_U$ and $A_T$, there is a coset representative of $g_2A_U$ which is in $A_T$. Without loss of generality, assume $g_2$ is this coset representative. Similarly, by considering the edge between $A_U$ and $g_1 A_T$, we may assume $g_1\in A_U$. 
	
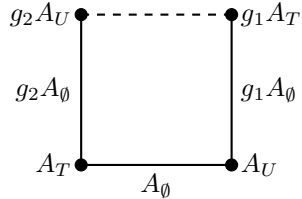
	\begin{center}
		\begin{tikzpicture}
		\def \lilRad {0.08}
		\draw [color=black,  style=thick] (0,0)--(0,2);
		\draw [color=black,  style=thick] (0,0)--(2,0);
		\draw [color=black,  style=thick, dashed] (0,2)--(2,2);
		\draw [color=black,  style=thick] (2,0)--(2,2);
		\draw [fill=black] (0,0) circle (\lilRad);
		\draw [fill=black] (0,2) circle (\lilRad);
		\draw [fill=black] (2,0) circle (\lilRad);
		\draw [fill=black] (2,2) circle (\lilRad);
		\node [left] at (0,0) {$A_T$};
		\node [right] at (2,0) {$A_U$};
		\node [left] at (0,2) {$g_2 A_U$};
		\node [right] at (2,2) {$g_1 A_T$};
		\node[below] at (1,0) {$A_\emptyset$};
		\node[right] at (2,1) {$g_1 A_\emptyset$};
		\node[left] at (0,1) {$g_2 A_\emptyset$};
		\end{tikzpicture}
	\end{center}

	If there was an edge connecting the two endpoints of this path, the cosets $g_2 A_U$ and $g_1 A_T$ would have nonempty intersection, implying there are elements $u\in A_U$ and $t\in A_T$ such that $g_2 u = g_1 t$. But this contradicts Lemma \ref{TU not equal UT}, so the endpoints of the path cannot be connected by an edge. Thus there are no full cycles of length 4 in $\widehat{X_\Gamma}$. Hence if $\widehat{\Gamma}$ has two connected components, then $\widehat{X_\Gamma}$ is 6-large.
\end{proof}

The following proposition of Januszkiewicz-Swaitkowski \cite{januszkiewicz_swiatkowski_2006}, specialized for our case, will be used in the last step of our proof of Theorem \ref{2-complete Artin cplx is systolic}. Recall that a simplicial complex is systolic if it is connected, simply-connected, and locally 6-large. 

\begin{prop}\cite[Prop 1.4]{januszkiewicz_swiatkowski_2006} \label{Jan-Swait}
	If $X$ is a systolic simplicial complex, then $X$ is $6$-large. 
\end{prop}

\theoremstyle{plain}
\newtheorem*{2-complete Artin cplx is systolic}{Theorem \ref{2-complete Artin cplx is systolic}}
\begin{2-complete Artin cplx is systolic}
	Let $A_\Gamma$ be a locally reducible Artin group. If there are at least three connected components in $\widehat{\Gamma}$, then $\widehat{X_\Gamma}$ is systolic. 
\end{2-complete Artin cplx is systolic}
	
\begin{proof}
	Let $A_\Gamma$ be a locally reducible Artin group and suppose $\widehat\Gamma$ has at least three connected components. By Proposition \ref{basic_props_2complete_artin_cplx}, the complex $\widehat{X_\Gamma}$ is connected and simply connected. To show that $\widehat{X_\Gamma}$ is systolic, we must show that it is locally 6-large. We will use induction on the number of connected components $n$ of $\widehat{\Gamma}$. 
	
	Start by assuming that $n=3$. To show that $\widehat{X_\Gamma}$ is locally 6-large, we need to show that \linebreak $sys(Lk_{\widehat{X_\Gamma}}(g\Delta_T))\geq 6$ for any simplex $g\Delta_T$. Choose an arbitrary simplex $g\Delta_T$ of type $T$, that is, choose a simplex in the orbit of a face of the fundamental domain $K$ which has local group $A_T$. By Proposition \ref{basic_props_2complete_artin_cplx}, $Lk_{\widehat{X_\Gamma}}(g\Delta_T)$ is isomorphic to the $2$-complete Artin complex associated with Artin group $A_T$. 
	
	By construction, $T$ is a set of standard generators determined by some collection of connected components of $\widehat\Gamma$. By the definition of $\widehat{X_\Gamma}$, $T$ cannot contain all of the standard generators: it must be missing all generators from at least one connected component of $\widehat\Gamma$. Consider the subgraph $Z$ of $\Gamma$ consisting only of the vertices associated to generators in $T$. Let $\widehat Z$ be the analogue of $\widehat \Gamma$, i.e., the graph with vertex set $T$ and all edges between vertices which were labeled 2 in $\Gamma$. Since the generators in at least one connected component of $\widehat\Gamma$ had to be omitted from $T$, $\widehat Z$ has strictly fewer connected components than $\widehat\Gamma$. 
	
	Let $k$ be the number of connected components of $\widehat Z$. By the assumption that $n=3$, we have $0\leq k<3$. If $k=2$, then by Lemma \ref{2 conn comps 6 large}, $\widehat {X_Z} \simeq Lk_{\widehat{X_\Gamma}}(g\Delta_T)$ is 6-large. If $k = 0$ or $1$, then $\widehat {X_Z} \simeq Lk_{\widehat{X_\Gamma}}(g\Delta_T)$ has no embedded full cycles, so $sys(Lk_{\widehat{X_\Gamma}}(g\Delta_T) = \infty$. Hence when $n=3$, $\widehat{X_\Gamma}$ is locally 6-large and therefore is systolic. 
	
	Now assume that $n>3$, i.e., $\widehat\Gamma$ has more than 3 connected components. Also assume that every 2-complete Artin complex $\widehat{X_Z}$ where $\widehat{Z}$ has $k<n$ connected components is 6-large. Indeed, if $k = 0, 1,$ or $2$, then we have shown $\widehat{X_Z}$ is 6-large. If $3\leq k<n$, then by induction we assume $\widehat{X_Z}$ is systolic and by Proposition \ref{Jan-Swait} it is 6-large. It remains to show that $\widehat{X_\Gamma}$ is locally 6-large. 
	
	Consider an arbitrary simplex $h\Delta_U$ of $\widehat{X_\Gamma}$ of type $U$. By Proposition \ref{basic_props_2complete_artin_cplx}, we know that $Lk_{\widehat{X_\Gamma}}(h\Delta_U)$ is isomorphic to the $2$-complete Artin complex associated with Artin group $A_U$. By construction, $U$ cannot contain generators from all the connected components of $\widehat\Gamma$, so as a $2$-complete Artin complex $Lk_{\widehat{X_\Gamma}}(h\Delta_U)$ is associated with strictly fewer than $n$ connected components. So by the induction hypothesis, $Lk_{\widehat{X_\Gamma}}(h\Delta_U)$ is 6-large. Since our choice of simplex was arbitrary, $\widehat{X_\Gamma}$ is locally 6-large. 
	
	Hence if there are at least three connected components of $\widehat\Gamma$, $\widehat{X_\Gamma}$ is systolic. 
\end{proof}

\section{Weakly Malnormal}\label{Weakly Malnormal Section}

In this section, we will use the systolicity of the 2-complete Artin complex and the following well-known result from systolic geometry to prove that most parabolic subgroups of locally reducible Artin groups are weakly malnormal. (See, e.g.,  \cite[Lemma 14]{cmv_large_2023}.) 

\begin{lemma}\label{systolic fixing}
	Let $G$ be a group acting without inversions on a systolic simplicial complex $X$. If $H\leq G$ fixes two vertices of $X$, then $H$ fixes pointwise every combinatorial geodesic between them. 
\end{lemma}

A combinatorial geodesic between two vertices is a minimal-length edge path. Note that in a systolic simplicial complex, there could be many distinct combinatorial geodesics between two given vertices. 

\begin{defn}
	Let $G$ be a group. A subgroup $H$ of $G$ is \textbf{weakly malnormal} if there exists $g\in G$ such that $|H \cap gHg^{-1}| <\infty$. 
\end{defn}

\theoremstyle{plain}
\newtheorem*{weakly malnormal}{Theorem \ref{weakly malnormal}}
\begin{weakly malnormal}
	Let $A_\Gamma$ be a locally reducible Artin group such that $\widehat{\Gamma}$ has at least two connected components. Then any parabolic subgroup $P$ of $A_\Gamma$ which has a 2-completion that is not all of $A_\Gamma$ is weakly malnormal. 
\end{weakly malnormal}

\begin{proof}
	Let $V$ be an arbitrary connected component of $\widehat{\Gamma}$. We will start by showing that the parabolic subgroup $A_{\Gamma\setminus V}$ is weakly malnormal. Our goal is to construct an element $g\in A_\Gamma$ such that $|A_{\Gamma\setminus V} \cap g A_{\Gamma\setminus V} g^{-1}|<\infty$; in fact, we will show this intersection is trivial. 
	
	Recall that by the construction of $\widehat{X_\Gamma}$, there is a vertex $v$ in the fundamental domain $K$ whose local group is $A_{\Gamma\setminus V}$. Furthermore, there is a face of $K$ of dimension $(\dim K -1)$ ``opposite'' to $v$ whose local group is $A_V$. Specifically, this face is determined by vertices associated to parabolic subgroups of the form $A_{\Gamma\setminus T}$ for $T\neq V$. Since $v$ is the only vertex of $K$ not contained in this face, $v$ is adjacent to every vertex in the face corresponding to $A_V$.
	
	Let the generators of $V$ be labeled $\{v_1, \dotsc, v_\ell\}$ and let $g = v_1 \dotsc v_\ell$. We will look at the translate of $K$ by this element $g$. By construction, $g$ setwise fixes the face of $K$ corresponding to $A_V$ and translates vertex $v$ (corresponding to $A_{\Gamma\setminus V}$) to the vertex $w = gv$ corresponding to $gA_{\Gamma\setminus V}$. Similarly to $v$, $w$ is connected to all vertices in the face of $K$ corresponding to parabolic subgroups of the form $A_{\Gamma\setminus T}$ for $T\neq V$. Hence the combinatorial distance between vertices $v$ and $w$ is 2. 
	
	In the case that $\widehat{\Gamma}$ has three or more connected components, any choice of vertex of $K$ on the face $A_V$ corresponds to a unique combinatorial geodesic between $v$ and $w$. By Theorem \ref{2-complete Artin cplx is systolic}, $\widehat{X_\Gamma}$ is systolic. Thus a subgroup that fixes two vertices must pointwise fix every combinatorial geodesic between them by Lemma \ref{systolic fixing}. Since $\text{Stab}(v) = A_{\Gamma\setminus V}$ and $\text{Stab}(w) = gA_{\Gamma\setminus V} g^{-1}$, every combinatorial geodesic between these vertices is fixed by $\text{Stab}(v) \cap \text{Stab}(w) = A_{\Gamma\setminus V} \cap gA_{\Gamma\setminus V} g^{-1}$. Hence $A_{\Gamma\setminus V} \cap gA_{\Gamma\setminus V} g^{-1}$ pointwise fixes all the vertices of $K$, so it must pointwise fix $K$ itself. The local group of $K$ is $A_\emptyset$ and thus has trivial stabilizer, so $A_{\Gamma\setminus V} \cap g A_{\Gamma\setminus V} g^{-1}$ must be trivial. 
	
	In the case that $\widehat{\Gamma}$ has exactly two connected components, $K$ is an edge, and the face of $K$ corresponding to $A_V$ is itself a vertex. In that case we have an edge path of length 2: $v = A_{\Gamma\setminus V}$ to $A_V$ to $w = gA_{\Gamma\setminus V}$. By Lemma \ref{2 conn comps 6 large}, $\widehat{X_\Gamma}$ is one-dimensional and 6-large. Thus this path of length 2 from $v$ to $w$ must be the unique combinatorial geodesic between $v$ and $w$, else $\widehat{X_\Gamma}$ would have a cycle of length 4. Since this geodesic is unique, it is fixed by $\text{Stab}(v)\cap \text{Stab}(w)$. Then as above,  $A_{\Gamma\setminus V} \cap g A_{\Gamma\setminus V} g^{-1}$ must pointwise fix $K$, and so the intersection must be trivial.

	Hence in both cases, $A_{\Gamma\setminus V}$ is weakly malnormal. Thus any subgroup of $A_{\Gamma\setminus V}$ is weakly malnormal, and any conjugate is also weakly malnormal. Hence if $P$ is a parabolic subgroup with a 2-completion that is not all of $A_\Gamma$, then $P$ is weakly malnormal. 
\end{proof}

\section{Acylindrical Hyperbolicity}\label{Acy Hyp Section}

	The following are theorems of Martin and Minasyan and Osin, respectively, which we will use to apply our result about weakly malnormal subgroups and show that certain locally reducible Artin groups are acylindrically hyperbolic. 

	\begin{thm}\cite[Theorem B]{martin_acy_hyp} \label{martin criterion}
		Let $X$ be a CAT(0) simplicial complex, with an action by simplicial isomorphisms of a group $G$. Assume there exists a vertex $v$ of $X$ such that: 
		\begin{enumerate}
			\item[(1)] The orbits of $G$ on the link $Lk_X(v)$ are unbounded for the associated angular metric. 
			\item[(2)] $Stab(v)$ is weakly malnormal in $G$. 
		\end{enumerate}
		Then $G$ is either virtually cyclic or acylindrically hyperbolic. 
	\end{thm}

	\begin{thm}\cite[Corollary 2.2]{minasyan_osin_acy_hyp_amalg} \label{acy_hyp_amalg}
		Let $G$ split as an amalgamated product: $G = A\ast_C B$ such that $A\neq C\neq B$ and $C$ is weakly malnormal in $G$. Then $G$ is either virtually cyclic or acylindrically hyperbolic.
	\end{thm}

	\theoremstyle{plain}
	\newtheorem*{acy hyp}{Theorem \ref{acy hyp}} 
	\begin{acy hyp} Let $A_\Gamma$ be an Artin group. If either
			\begin{enumerate} 
				\item $A_\Gamma$ is locally reducible and has a maximal finite-type subgroup which is dihedral and which has a 2-completion that is not all of $A_\Gamma$, or
				\item $A_\Gamma$ splits as an amalgamated product $A_{\Gamma_1} \ast_{A_{\Gamma_1 \cap \Gamma_2}} A_{\Gamma_2}$ such that $A_{\Gamma_1}$ is locally reducible and there is a 2-completion of $A_{\Gamma_1 \cap \Gamma_2}$ which does not contain all of $A_{\Gamma_1}$,
			\end{enumerate}
		then $A_\Gamma$ is acylindrically hyperbolic.
	\end{acy hyp}

	\begin{proof}
		Let $A_\Gamma$ an Artin group. 
		
		(1): We assume that $A_\Gamma$ is locally reducible and has a maximal finite-type subgroup which is dihedral and which has a 2-completion that is not all of $A_\Gamma$. Thus $\Gamma$ contains at least one edge, say $[a,b]$, with $m_{ab}\geq 3$, such that $A_{\{a,b\}}$ is a maximal finite-type standard parabolic subgroup (i.e., $A_{\{a,b\}}$ is not contained in a larger finite-type parabolic subgroup) and there is a 2-completion of $A_{\{a,b\}}$ that is not all of $A_\Gamma$.  
		
		Since $A_\Gamma$ is locally reducible, the Deligne complex $D_\Gamma$ associated to $A_\Gamma$ is CAT(0) by \cite{charney_tits_2000}. Consider $Lk_{D_\Gamma}(A_{\{a,b\}})$. There is a path in the link given by the sequence $\{g_i A_\emptyset \}_{i=0}^{\infty}$ where $g_i = ababa\dotsc$ is a word of length $i$. By \cite[Lemma 4.5]{vaskou_acy_hyp_2diml} the length of $g_i$ is unbounded, and by Theorem \ref{weakly malnormal}, $A_{\{a,b\}}$ is weakly malnormal. Hence by Theorem \ref{martin criterion}, $A_\Gamma$ is acylindrically hyperbolic. 
		
		(2): Now we assume that $A_\Gamma$ splits as an amalgamated product $A_\Gamma = A_{\Gamma_1} \ast_{A_{\Gamma_1\cap \Gamma_2}} A_{\Gamma_2}$ such that $A_{\Gamma_1}$ is locally reducible and there is a 2-completion of $A_{\Gamma_1 \cap \Gamma_2}$ which does not contain all of $A_{\Gamma_1}$. By Theorem \ref{weakly malnormal}, $A_{\Gamma_1 \cap \Gamma_2}$ is weakly malnormal in $A_{\Gamma_1}$, and hence is weakly malnormal in $A_\Gamma$. By Theorem \ref{acy_hyp_amalg}, $A_\Gamma$ is either virtually cyclic or acylindrically hyperbolic. Since $A_\Gamma$ is an amalgamated product, it is not virtually cyclic, hence it is acylindrically hyperbolic. 
	\end{proof}

	\begin{rmk}
		One may wonder whether the assumption that $\widehat{\Gamma}$ has two connected components is necessary. There are certainly counterexamples to acylindrical hyperbolicity when $\widehat{\Gamma}$ has one connected component. Take, for example, an Artin group which splits as a direct product. Then $\widehat{\Gamma}$ has one connected component, and $A_\Gamma$ is not acylindrically hyperbolic. 
	\end{rmk}

\section{Future Directions}\label{Future Directions Section}

We can ask whether the intersection of a collection of parabolic subgroups of a particular Artin group is itself a parabolic subgroup. This is another question that is open in general, but is conjectured to be true for all Artin groups. It has been proven to be true for several classes including finite-type Artin groups \cite{cumplido_et_al_2019}, right-angled Artin groups (Artin groups for which $m_{ij} = 2$ or $\infty$ for all $i,j$) \cite{duncan_kaz_rem_2007}, large-type Artin groups (Artin groups for which $m_{ij}\geq 3$ for all $i,j$) \cite{cmv_large_2023}, and 2-dimensional, (2,2)-free Artin groups \cite{blufstein_2022}. It is also known to be true for the intersection of two parabolic subgroups of an FC-type Artin group, when at least one of the parabolic subgroups is finite-type \cite{morris-wright_2021} and \cite{moller_paris_varghese_2022}. 

In many cases when the affirmative answer is known, the solution has come from the study of one or more geometric spaces with some property of non-positive curvature (e.g. hyperbolic, CAT(0), systolic). In this case, when stabilizers of vertices (or simplices, cubes, etc.) are parabolic subgroups, the geometric properties of the action help greatly in the study of their intersection. Examples include the Deligne complex and the Artin complex (see Section \ref{Deligne and Artin Complexes Section}).

\begin{question}
	For $A_\Gamma$ a locally reducible Artin group, is the intersection of an arbitrary collection of $2$-complete parabolic subgroups itself a parabolic subgroup? 
\end{question}

\newpage
\printbibliography

\end{document}